\documentclass[11pt]{amsart} 
\usepackage{amsmath,amssymb} 
\usepackage{color}
\usepackage[latin1]{inputenc}
\usepackage{color}
\usepackage[dvipsnames]{xcolor}

\newtheorem{theorem}{Theorem} 
\newtheorem{lem}[theorem]{Lemma} 
\newtheorem{prop}[theorem]{Proposition} 
\newtheorem{rem}[theorem]{Remark} 

\newtheorem{defi}[theorem]{Definition}

\newcommand{\cqfd}{{\nobreak\hfil\penalty50\hskip2em\hbox{}\nobreak\hfil
$\square$\qquad\parfillskip=0pt\finalhyphendemerits=0\par\medskip}}

\newcommand{\R}{{\mathbb R}}

\newcommand{\rt}{{\rho_T}} 
\newcommand{\into}{{\int_\Omega}} 
\newcommand{\mo}{{\mu(\rho)}} 
\newcommand{\mop}{{\mu'(\rho)}} 
\newcommand{\al}{{\sqrt{2r}}}

\DeclareMathOperator*{\dive}{div}

\title[Long time behaviour]{On the Exponential decay for Compressible Navier-Stokes-Korteweg equations\\ with a Drag Term}

\author[D. Bresch]{D. Bresch}
\address{Laboratoire de Math\'ematiques, CNRS UMR 5127,
Universit\'e  Savoie Mont-Blanc, 73376 Le Bour\-get-du-Lac, France;\\
\rm  email: Didier.Bresch@univ-smb.fr}
 
\author[M. Gisclon]{M. Gisclon}
\address{Laboratoire de Math\'ematiques, CNRS UMR 5127,
Universit\'e Savoie Mont-Blanc, 73376 Le Bour\-get-du-Lac, France; \\
\rm email: gisclon@univ-smb.fr}

\author[I. Violet]{I. Lacroix-Violet} 
\address{Laboratoire de Math\'ematiques, CNRS UMR 8524 
Universit\'e de Lille, 59655 Villeneuve d'Ascq, France; Team Rapsodi INRIA Lille Nord Europe
  \rm e-mail: ingrid.violet@univ-lille1.fr}

\author[A. Vasseur]{A. Vasseur} 
\address{Department of Mathematics, University of Texas at Austin, 2515 Speedway Stop C1200 Austin, TX 78712, USA;\\
\rm email: vasseur@math.utexas.edu}

\begin{document}

\maketitle

\begin{abstract}

In this paper, we consider global weak solutions to compressible Navier-Stokes-Korteweg equations with density dependent viscosities, in a periodic domain $\Omega = \mathbb T^3$, with a linear drag term with respect to the velocity. The main result concerns the exponential decay  to equilibrium of such solutions using log-sobolev type inequalities. In order to show such a result, the starting point is  a global weak-entropy solutions definition,  introduced in {\sc  D. Bresch, A. Vasseur} and {\sc C.~Yu} \cite{BrVaYu}. Assuming extra assumptions on the shear viscosity when the density is close to vacuum and when the density tends to infinity, we conclude the exponential decay to equilibrium. Note that our result covers the quantum Navier-Stokes system with a drag term. 
\end{abstract}

\medskip




\section{Introduction}
  The goal of this paper is to study the long-time behaviour of solutions of Navier-Stokes-Korteweg models with a drag term of type $r_2  \, \rho   \, u$ (for $r_2 > 0$ a constant) in a periodic domain $\Omega= {\mathbb T}^3$ using  a $\kappa$ relative entropy generalizing the one by \cite{BrGiLa}. 
   Without loss of generality we fix $|\Omega|=1$.
   More precisely the system under consideration reads
\begin{eqnarray}
&&\partial_t\rho + {\rm div}(\rho u) = 0, \nonumber \\ 
&&\partial_t (\rho u)   + {\rm div}(\rho u\otimes u) + \nabla p(\rho) \nonumber \\
&& \hskip2cm - 2  \, {\rm div}(\mu(\rho) {\mathbb D}(u)) 
     - \nabla (\lambda(\rho){\rm div} \, u) \nonumber \\
 && \hskip2cm - 2 \,  \varepsilon  \,  \rho \,  \nabla \Bigl(\sqrt{K(\rho)} \Delta
         \int_0^\rho \sqrt{K(s)} \, ds\Bigr)    + r_2 \,  \rho \, u = 0\nonumber      
\end{eqnarray} 
where $\varepsilon >0$ is the Planck constant and ${\mathbb D}(u)=\left(\nabla u+ {}^t\nabla u\right)/2$ stands for the symmetric part of the velocity gradient. 
The pressure state law reads $p(\rho)= a \, \rho^\gamma$ where $a>0$
and $\gamma \geq 1$ are constants. 
The shear and bulk viscosities are assumed to satisfy the BD relation
\begin{equation}\label{BDrelation}
\lambda(\rho) = 2(\mu'(\rho)\rho -\mu(\rho)).
\end{equation}
The shear viscosity and the capillarity coefficient are linked through the relation  $K(\rho)= (\mu'(\rho))^2/\rho$.

Readers interested by Korteweg type systems are referred to the following articles and books: \cite{Ko, Va, CaHi, DuSe, RoWi, Ni, HeMa, AnMa, AnMa1, AnSp1} and references cited therein. Note that the Korteweg type models contain the class of quantum models which corresponds to $K(\rho)=1/\rho$ which gives $\mu(\rho)=\rho$ and, with the BD relation \eqref{BDrelation}, $\lambda(\rho)=0$. 

Quantum fluid models have attracted a lot of attention in the last decades due to the variety of applications. Indeed, such models can be used to describe superfluids \cite{LoMo}, quantum semiconductors \cite{FeZho}, weakly interacting Bose gases \cite{Gr} and quantum trajectories of Bohmian mechanics \cite{Wy}. Recently some dissipative quantum fluid models have been derived. In par\-ticular,  under some assumptions and using a Chapman-Enskog expansion in Wigner equation, the authors have obtained in \cite{BruMe09} the so-called quantum Navier-Stokes model. Roughly spea\-king, it corresponds to the classical Navier-Stokes equations with a quantum correction term.  The main difficulties of such models lie in the highly nonlinear structure of the third order quantum term and the proof of positivity (or non-negativity) of the particle density. 

In \cite{VaYu1}, A. Vasseur and C. Yu have proven the global-in-time existence of weak solutions of the quantum Navier-Stokes equations with two drag terms of type $r_0  \, u$ and $r_1 \,  \rho |u|^2 \,  u$.  Their result is still valuable in the case $r_1=0$ and their proof is based on a Faedo-Galerkin ap\-pro\-xi\-ma\-tion (following the ideas of \cite{Ju}) and the BD entropy (see \cite{BrDe2003, BrDeLi2003}). Note that the result is also still valuable with the add of a third drag force term of type $r_2 \,  \rho  \, u$ because the term does not perturb the uniform estimates and stable in 
limit process. In the same time, M. Gisclon and I. Lacroix-Violet \cite{GiLa} also obtained the existence of global-in-time weak solutions for the same model using a cold pressure instead of drag terms with the same method of Faedo-Galerkin approximation. In \cite{VaYu2}, the authors use the result obtained in  \cite{VaYu1} and pass to the limits $\varepsilon, r_0, r_1$ tend to zero to prove the existence of global-in-time weak solutions to degenerate compressible Navier-Stokes Equations with $\mu(\rho)=\mu \,  \rho$ and $\lambda(\rho)=0$. 
Such existence of global existence of weak solutions has also been obtained at the same time and independently in \cite{LiXi}.  Note that to prove the result in \cite{VaYu2}, they need uniform (with respect to $r_0, r_1$) estimates to pass to the limit $r_0, r_1$ tend to $0$. To this end they have to firstly pass to the limit $\varepsilon$ tends to $0$. 

Recently in \cite{LaVa}, global existence of weak solutions for the quantum Navier-Stokes
 Equations has been proved without drag terms. The method is based on the construction of weak solutions that are renormalized in the velocity variable for the system with drag terms $r_0 u$ and $r_1 \rho |u|^2 u$ and uses the stability part of the result to pass to the limit $r_0, r_1$ tend to zero to obtain weak renormalized solutions of the system without drag terms. All weak renormalized solution being a weak solution, the existence of global weak solutions for the quantum Navier-Stokes system without drag terms is shown. Adding an extra drag term $r_2 \,  \rho \, u$ to the system does not change the result: it does not perturb the estimates and is stable in limit passage. Note that the construction being uniform with respect to the Planck constant $\varepsilon$, the authors also perform the semi-classical limit $\varepsilon \to 0$ to the associated compressible Navier-Stokes equations. It is important to remark that a global weak solution of the quantum Navier-Stokes equations in the sense of \cite{LaVa} is also weak solution of the corresponding augmented system as introduced in \cite{BrDeZa}.  
Note also the paper \cite{AnSp} concerning the global existence for the quantum Navier-Stokes system using the approximate procedure introduced by J. Li and Z.P. Xin (see\cite{LiXi}). 

We also mention the recent paper \cite{AnSp1} where they consider the global existence of weak solutions to Navier-Stokes-Korteweg model extending the method of truncation, regularization and renormalization introduced by A. Vasseur and C. Yu in \cite{VaYu2}  and generalized by I.~Lacroix-Violet and A. Vasseur in \cite{LaVa}.

Finally, let us cite \cite{BrVaYu} where the existence of global weak solutions for some Navier-Stokes-Korteweg system in $\Omega= {\mathbb T}^3$ is obtained. The proof is performed around four main ideas. The first one is to introduce the equation satisfied by $\mu(\rho)$ in the definition of weak solutions assuming $\mu(\rho_0)$ in $L^1(\Omega)$ to relax the constraints on the pressure exponent. The second one consists in extending the two-velocities framework, introduced by D. Bresch, B. Desjardins and E. Zatorska in \cite{BrDeZa}, with capillarity and more general drag terms using a generalization of the quantum B\"ohm identity developed in \cite{BrCoNoVi}. The third one is the proof of a generalized high-order log-Sobolev dissipation inequality similar to the one used in \cite{Ju} and established by A. J\"ungel and D. Matthes in \cite{JuMa}. Finally the last one is the extend of the renormalized framework, introduced by I.~Lacroix-Violet and A. Vasseur in \cite{LaVa} for the quantum case, to viscosities satisfying the BD relation. 
Note that in this paper we use their definition of weak solution that we recall, for reader convenience, in the next section (see Definition \ref{defweak1}). 

Denoting by $r$ the mean value of $\rho$ on $\Omega$ {\it i.e.} $$r=\into \rho$$ the goal of the paper is to show that a  global weak solution $(\rho,u)$  in the sense of \cite{BrVaYu} is exponentially decaying in time to $(r,0)$. To obtain such a result we use the augmented formulation of the system firstly introduced in \cite{BrDeZa} and we use a relative entropy version of the $\kappa$ entropy similar to the one introduced in \cite{BrGiLa}.

    
\vskip0.1cm    
The paper is organized as follows. In Section \ref{AssMainRe} we give the definition of weak solution used in the paper, all the assumptions and we state the main result of exponential decay (see Theorem \ref{th:exponentialdecay}). Section \ref{sec_NSK} is devoted to the proof of this exponential time decay. Finally we give in Appendix the proof of a technical lemma (Lemma \ref{lemtech}) used in Section \ref{sec_NSK}. Note that in all the paper the notation $C$ is used for constants independent of time $t$ but which can be different from line to line. 


\section{Assumptions and main result} \label{AssMainRe}

In this section, we first recall the definition of weak solution used in \cite{BrVaYu}. We complete this definition with some assumptions on the shear and bulk viscosities and finally we state the main result of the paper. 

\subsection{Definition of weak solutions}

Let us first recall the definition of weak solution used in \cite{BrVaYu}. The authors consider shear and bulk viscosities such that:

\begin{itemize}
	\item denoting by ${\mathbb R_+}= [0,\infty)$ and ${\mathbb R}_+^* = (0,\infty)$. 
	\begin{equation}\label{regmu}
\mu \in {\mathcal C}^0({\mathbb R}_+; {\mathbb R}_+) 
       \cap {\mathcal C}^2({\mathbb R}_+^*, {\mathbb R}),
       \end{equation}
	\item there exists two positive numbers $\alpha_1, \alpha_2$ satifying $2/3 <\alpha_1<\alpha_2<4$ such that for any $\rho>0$
	\begin{equation}
	\label{hyp1} 
	0 < \frac{1}{\alpha_2} \rho  \, \mu'(\rho) \le \mu(\rho) \le \frac{1}{\alpha_1} \rho  \, \mu'(\rho), 
	\end{equation}
	\item there exists a constant $C>0$ such that
	\begin{equation}\label{hyp2}
	\Bigl|\frac{\rho \,  \mu''(\rho)}{\mu'(\rho)}\Bigr| \le C <+\infty.
	\end{equation}
\end{itemize}

Note that, thanks to \eqref{hyp1} and \eqref{hyp2}, there exists $\widetilde \nu>0$ such that
\begin{equation}\label{prop1}
\lambda(\rho) + \dfrac{2}{3} \mu(\rho) \ge \widetilde\nu \mu(\rho)
\end{equation}
and 
$$\mu(0)=\lambda(0)=0.$$  
The method, used in \cite{BrVaYu},  to prove global existence of weak solution to the Navier-Stokes-Korteweg  is linked
to the two-hydrodynamic system introduced in \cite{BrDeZa}, the extension of the B\"ohm identity
proved in \cite{BrCoNoVi}, the  generalization of the dissipation inequality used in \cite{Ju} for 
the quantum Navier-Stokes system and established in \cite{JuMa} and 
the renormalized solutions introduced in \cite{LaVa}.  

Here, the addition of a drag term $r_2 \, \rho  \, u$ does not change the result because the $\kappa$-entropy remains uniform and the term $\rho \, u$ is already in the time derivative for compactness. 

\begin{rem} Note the recent paper \cite{AlBr} where the authors have proved a Sobolev 
inequality which could help to enlarge the range of viscosities namely 
$$\mu(\rho) = \rho^\alpha, \qquad \lambda(\rho) = 2(\alpha-1) \rho^\alpha$$
with $\dfrac{2}{3}< \alpha <+\infty$. 
\end{rem}

Let us now recall the definition of the global weak solutions for the Navier-Stokes-Korteweg system
similar to what has been developped in \cite{BrVaYu}
\begin{defi}\label{defweak1}  
    We say that $(\rho,u)$ is a global weak solution to the compressible Navier-Stokes-Korteweg equations as constructed in \cite{BrVaYu} if
 \begin{itemize}
\item The density satisfies $\rho \in  {\mathcal C}^0([0,+\infty); L^\gamma_{\rm weak}(\Omega))$
with $\rho \ge 0 \hbox{ in } (0,+\infty)\times \Omega$, $\rho\vert_{t=0} = \rho^0 \hbox{ a.e. in } \Omega$ with
the viscosity $\mu(\rho) \in {\mathcal C}^0([0,+\infty); L_{\rm weak}^{3/2}(\Omega))$.
\item  The momentum  satisfies $\rho\, u \in {\mathcal C}([0,+\infty); L^{2\gamma/(\gamma+1)}_{\rm weak}(\Omega))$ 
with $\rho u\vert_{t=0} = m^0$.
\item The following energy estimate holds {\it a.e} t $\in [0,+\infty)$
\begin{eqnarray}\label{EnergyEstimates}
&& \int_\Omega
\rho \Bigl(\frac{|w|^2}{2} + [(1-\kappa)\kappa+ \varepsilon] \frac{|v|^2}{2}\Bigr)
      + H(\rho)+ 2 \, r_2 \,  \kappa \,h(\rho)  \, \nonumber\\
&& + \tilde \nu \int_0^t \int_\Omega |{\mathbb T}_\mu|^2  \, 
    + r _2 \int_0^t\int_\Omega \rho |u|^2 \\
 && +  \varepsilon \, C_{\alpha_1,\widetilde\nu} 
      \int_0^t\int_\Omega \Bigl[|\nabla^2 Z(\rho)|^2 + |\nabla Z_1(\rho)|^4
     + |{\mathbb T}_\varepsilon|^2\Bigr]   \,   \nonumber \\
&& + 2\kappa \int_0^t \int_\Omega \frac{\mu'(\rho)p'(\rho)}{\rho} |\nabla\rho|^2 \,
     \nonumber \\
&& \le   \int_\Omega \Bigl[\rho^0 \Bigl(\frac{|w_0|^2}{2} + [(1-\kappa)\kappa+ \varepsilon] \frac{|v_0|^2}{2}
\Bigl)
      + H(\rho^0)+ 2 \,  r_2 \,   \kappa \, h(\rho^0)\Bigr] \,\nonumber
\end{eqnarray}
 for some $C_{\alpha_1,\widetilde\nu}>0$ depending on $\alpha_1$ and $\widetilde \nu$ and
some  $\kappa\in (0,1)$ with 
 \begin{eqnarray}
&& w=u+\kappa \, v \quad  \hbox{ and } \quad v=2\nabla s(\rho) \quad
    \hbox{ where } \quad s'(\rho)= \dfrac{\mu'(\rho)}{\rho}  \label{vw}\\
&& H(\rho)=\rho \int_1^\rho \dfrac{p(s)}{s^2}\,  \quad \hbox{ and }  \quad h(\rho) =\rho \int_1^\rho \dfrac{\mu(s)}{s^2} \,  \label{hH}  \\
&& Z(\rho)= \int_1^\rho \frac{\sqrt{\mu(s)} \mu'(s)}{s} \,  \quad
      \hbox{ and } \qquad
      Z_1(\rho) = \int_1^\rho \frac{\mu'(s)}{(\mu(s))^{1/4} s^{1/2}} \, 
 \end{eqnarray}
and  where ${\mathbb T}_\mu$ is defined through
\begin{eqnarray}\label{compat1}
\sqrt{\mu(\rho)} {\mathbb T}_\mu =
      \nabla \left( \sqrt\rho u \frac{\mu(\rho)}{\sqrt\rho} \right) - \sqrt\rho u \otimes \sqrt{\rho} \nabla s(\rho)
\end{eqnarray}
\begin{eqnarray}\label{compat2}
&& \frac{\lambda(\rho)}{2\mu(\rho)} {\rm Tr}(\sqrt{\mu(\rho)} {\mathbb T}_\mu) {\rm Id} \\
&& \hskip1cm = 
\left[{\rm div} \left( \frac{\lambda(\rho)}{\mu(\rho)} \sqrt\rho u \frac{\mu(\rho)}{\sqrt\rho} \right) 
- \sqrt\rho u \cdot \sqrt\rho \nabla s(\rho) \frac{\rho \mu''(\rho)}{\mu'(\rho)}  \right]{\rm Id}.\nonumber
\end{eqnarray}
The same definitions and compatibility condition are satisfied for  ${\mathbb T}_\varepsilon$,
replacing $u$ by $v=2\nabla s(\rho)$ respectively  in \eqref{compat1}
and \eqref{compat2}.
  \item The following extra estimates hold
 \begin{eqnarray}
 && \|\mu(\rho)\|_{L^\infty(0,+\infty;W^{1,1}(\Omega))} 
     + \|\mu(\rho) u\|_{L^\infty(0,\infty;L^{3/2}(\Omega))\cap
         L^2(0,+\infty;W^{1,1}(\Omega))}<+\infty \nonumber\\
 && \|\partial_t \mu(\rho)\|_{L^\infty(0,+\infty; W^{-1,1}(\Omega))}+
     \|Z(\rho)\|_{L^{1+}(0,+\infty;L^1(\Omega))} < +\infty. \nonumber
 \end{eqnarray}
 \item The density $\rho$ satisfies the mass equation in  ${\mathcal D}'((0,+\infty)\times \Omega)$:
\begin{equation}\label{mass}
\partial_t \rho + {\rm div} (\rho u) = 0. 
\end{equation}
\item The velocity satisfies the momentum in  ${\mathcal D}'((0,+\infty)\times \Omega)$ :
\begin{eqnarray}
&& \partial_t(\rho u)  + {\rm div}(\rho u \otimes u) + \nabla p(\rho)  + r_2 \,  \rho  \, u \nonumber \\
&&\hskip2cm    - 2 {\rm div}(\sqrt{\mu(\rho)} {\mathbb S}_\mu
    + \frac{\lambda(\rho)}{2\mu(\rho)}  {\rm Tr} (\sqrt{\mu(\rho)} {\mathbb S}_\mu) {\rm Id}) 
     \label{momentum}  \\
&&\hskip2cm    - 2 \varepsilon  {\rm div}(\sqrt{\mu(\rho)} {\mathbb S}_\varepsilon
    + \frac{\lambda(\rho)}{2\mu(\rho)}  {\rm Tr} (\sqrt{\mu(\rho)} {\mathbb S}_\varepsilon) {\rm Id}) 
     \nonumber   \\
 && \hskip2cm = 0   \nonumber
\end{eqnarray}
with
$${\mathbb S}_\mu = \frac{({\mathbb T}_\mu + {\mathbb T}_\mu^t)}{2}, \qquad 
{\mathbb S}_\varepsilon = \frac{({\mathbb T}_\varepsilon 
                                            + {\mathbb T}_\varepsilon^t)}{2}.
$$
\item The viscosity satisfies in  ${\mathcal D}'((0,+\infty)\times \Omega)$: 
\begin{eqnarray}
\label{mu}
\partial_t \mu(\rho) + {\rm div}(\mu(\rho) u) 
+  \frac{\lambda(\rho)}{2\mu(\rho)} {\rm Tr}(\sqrt{\mu(\rho)}{\mathbb T}_\mu) = 0.
\end{eqnarray}
\item The global weak solution is obtained through regularized solutions $(\rho_\zeta, w_\zeta, v_\zeta)$
(letting $\zeta$ tend to zero) satisfying
\begin{eqnarray}
&& \dfrac{d\mathcal{E_\zeta}}{dt}(t)  + \varepsilon C_{\widetilde\nu,\alpha_1} \, \displaystyle \int_\Omega |\nabla^2 Z(\rho_\zeta)|^2 \nonumber \\
&&  \hskip3cm 
  + \kappa \int_\Omega\mu'(\rho_\zeta)H''(\rho_\zeta) |\nabla \rho_\zeta|^2 
 + \int_\Omega  r _2 \rho_\zeta |u_\zeta|^2\le 0 \label{dissip}
\end{eqnarray}
where 
\begin{equation}\label{Energy}
{\mathcal E}_\zeta(t) = \int_\Omega \Big[\rho_\zeta (|w_\zeta|^2 
  + (\kappa(1-\kappa)+\varepsilon) |v_\zeta|^2) + H(\rho_\zeta\vert r) 
   +2  \, r_2 \,  \kappa \,   h(\rho_\zeta\vert r)\Bigr]
\end{equation}
where $H(\rho_\zeta\vert r)$ and $h(\rho_\zeta\vert r)$ are  defined by:
\begin{eqnarray}
&& \nonumber H(\rho_\zeta\vert r) = H(\rho_\zeta) - H(r) - H'(r)(\rho_\zeta - r), \\
 &&\nonumber    h(\rho_\zeta\vert r) = h(\rho_\zeta)-h(r) - h'(r)(\rho_\zeta-r).
\end{eqnarray}
\end{itemize}
\end{defi}

\begin{rem}
The presence of the drag term $r_2 \,  \rho \, u$ with $r_2>0$
implies (using the $\kappa$-entropy estimate) the control of $\sup_{t\in (0,+\infty)} \displaystyle{\int_\Omega} h(\rho)(t,x)$ if $\displaystyle{\int_\Omega} h(\rho_0)$ initially. This reads
$$\sup_{t\in (0,+\infty)} \int_\Omega \rho \int_1^\rho \dfrac{\mu(s)}{s^2} \,  \le C <+\infty.$$
Using Hypothesis \eqref{hyp1}, this implies that 
$$\mu(\rho) \in L^\infty(0,+\infty; L^1(\Omega)).$$
This is an important information compared to previous papers such as \cite{LiXi} or \cite{BrVaYu}.
Indeed, to  control $\mu(\rho)$ in $L^\infty(0,+\infty; L^1(\Omega))$ without using drag terms information, we would need as in \cite{LiXi}  to consider a pressure law as
$$p(\rho) = a \, \rho ^\gamma \hbox{ with } \gamma \ge \max (2\alpha_2-1,1)$$
but this condition is restrictive. Note that $\rho^\gamma \in L^\infty(0,+\infty;L^1(\Omega))$ implies 
$\mu(\rho) \in L^\infty(0,+\infty;L^1(\Omega))$ because $\sqrt\rho \in L^\infty(0,+\infty; L^2(\Omega))$
deduced from the mass conservation. Note that the hypothesis  $\gamma \ge 2\alpha_2-1$ is not present in \cite{BrVaYu}. It suffices to assume  $\mu(\rho_0) \in L^1(\Omega)$ and use the equation satisfied by $\mu(\rho)$ to get the result $\mu(\rho)\in L^\infty(0,T;L^1(\Omega))$ for all $T<+\infty$
but this result is not valid for  $T=+\infty$. 
\end{rem}

\subsection{Main result: exponential decay}

In order to be able to get exponential decay, we add the following assumptions on the shear viscosity: 

\begin{itemize}
\item {\it Behaviour of $\mu(\rho)$ for $\rho$ small enough:
\begin{equation}\label{condalphamu}
\lim_{\rho\to0}\frac{\rho\mu'(\rho)}{\mu(\rho)}=\alpha,\qquad \mathrm{with} \ \ \frac{2}{3}<\alpha\le 4,
\end{equation}
assuming
\begin{itemize}
\item  If $\alpha <1$ or if there exists $\rho_0>0$ such that $\mu(\rho)\ge C\rho$ for $\rho<\rho_0$ \\
          that the pressure exponent satisfies $\gamma\ge 1$,
\item If $\alpha \ge 1$  that  the pressure exponent satisfies $\gamma<\alpha.$
\end{itemize}
\vskip0.2cm
\item {\it Behaviour of $\mu(\rho)$ for $\rho$ big enough:} 
\begin{equation}
\label{condbeta2mu}
\lim_{\rho\to\infty}\frac{\rho\mu'(\rho)}{\mu(\rho)}=\beta,\qquad \mathrm{with} \ \ 1< \beta<4
\end{equation}
or letting $M>0$ a given constant
\begin{equation}
\label{condbeta1mu}
\mu(\rho)= \rho \hbox{ for any } \rho \ge M.
\end{equation}
}
\end{itemize}

\begin{rem}
The quantum case satisfies \eqref{condalphamu}--\eqref{condbeta1mu}  without any restriction 
on the pressure law namely $p(\rho)= a \rho^\gamma$  with $\gamma \ge 1$ is elligible.
\end{rem}

As mention before, denoting by $r$ the mean value of $\rho$ on $\Omega$, the goal is to show that a  global weak solution $(\rho,u)$  in the sense of Definition \ref{defweak1} with $p(\rho)= a \, \rho^\gamma$ with $\gamma\ge 1$ and the shear and bulk viscosities satisfying \eqref{BDrelation}--\eqref{regmu} and \eqref{condalphamu}--\eqref{condbeta1mu} is exponentially decaying in time to $(r,0)$. More precisely, the goal is to prove the following result:
\begin{theorem}
\label{th:exponentialdecay} Let $(\rho,u)$ be a global weak solution of the compressible
Navier-Stokes-Korteweg System in the sense of Definition  {\rm \ref{defweak1}} with $p(\rho)= a \, \rho^\gamma$ with $\gamma\ge 1$ and the shear and bulk viscosities satisfying \eqref{BDrelation}--\eqref{regmu} and \eqref{condalphamu}--\eqref{condbeta1mu}. We  define the relative entropy
$$\mathcal{E}(\rho,u,v\vert r,0,0) = {\mathcal E}(t) = \int_\Omega [\rho (|w|^2+(\kappa(1-\kappa)+\varepsilon) |v|^2) 
 + H(\rho\vert r) + 2 \, \kappa \,  r _2\, h(\rho\vert r)],$$
where $v,w$ are given by \eqref{vw} and $H(\rho\vert r)$ and $h(\rho\vert r)$ are defined by
$$H(\rho\vert r) = H(\rho) - H(r) - H'(r) (\rho-r),\quad
    h(\rho\vert r) = h(\rho) - h(r) - h'(r) (\rho-r),$$
with $h$, $H$ given by \eqref{hH}. 

Then, for some constant $C>0$,
$${\mathcal E}(t)  \le {\mathcal E}(0) \exp(-C t).$$

\end{theorem}


\section{Proof of Theorem \ref{th:exponentialdecay}}\label{sec_NSK}
   The goal of this section is to prove Theorem \ref{th:exponentialdecay}. As global weak
solutions are constructed using regular enough solutions, the formal calculations
we will do may be justified using these approximate solutions and their initial conditions. We will then write first the exponential decay for $\mathcal{E}_\zeta$ and then pass to the limit
with respect to $\zeta$ to obtain Theorem \ref{th:exponentialdecay} assuming strong convergence initially. The goal of this section is then to prove the following Theorem.
\begin{theorem} \label{ineqentrel}  
Let the hypothesis on the viscosities and on the pressure law be satisfied namely $p(s)= a \, s^\gamma$ with $\gamma \ge 1$ and the shear and bulk viscosities satisfying \eqref{regmu}--\eqref{BDrelation} and \eqref{condalphamu}--\eqref{condbeta1mu}.
Then there exists a constant $C>0$, uniform in $\zeta$, such that for any  $t>0$, the quantity $\mathcal{E}_\zeta(t)$ defined in \eqref{Energy} verifies
\begin{eqnarray*}
&&   \mathcal{E}_\zeta(t) \leq \mathcal{E}_\zeta(0) \exp(-Ct).
\end{eqnarray*}
\end{theorem}
Note that thanks to \eqref{dissip} of Definition \ref{defweak1}, the regular approximation satisfies the following inequality  
\begin{eqnarray}
\dfrac{d\mathcal{E}_\zeta}{dt}(t) &\leq&  - \varepsilon  \, C_{\widetilde\nu,\alpha_1} \, \displaystyle \int_\Omega |\nabla^2 Z(\rho_\zeta)|^2 - \kappa \int_\Omega\mu'(\rho_\zeta)H''(\rho_\zeta) |\nabla \rho_\zeta|^2 
 - \int_\Omega  r_2 \,   \rho_\zeta \,  |u_\zeta|^2  \nonumber  \\ \label{dissip1}
\end{eqnarray}
It remains then to prove  the following proposition to obtain theorem \ref{ineqentrel} for the regular approximation.
\begin{prop}\label{prop:estregsol}
There exists a constant $C>0$ such that for every $t>0$
$$- \varepsilon  \, C_{\widetilde\nu,\alpha_1} \, \displaystyle \int_\Omega |\nabla^2 Z(\rho_\zeta)|^2 - \kappa \int_\Omega\mu'(\rho_\zeta)H''(\rho_\zeta) |\nabla \rho_\zeta|^2 
 - \int_\Omega  r _2 \,  \rho_\zeta \,  |u_\zeta|^2  \leq- C \mathcal{E_\zeta}(t).$$
\end{prop}

The proof of this proposition is divided in three steps corresponding to the three following subsections. To simplify the presentation, we omit the subscript $\zeta$ but all the calculations have to be thinking for the approximate solutions.

\subsection{First step of proof of proposition \ref{prop:estregsol}}

First of all, let us prove this following lemma.
\begin{lem}\label{Result_part1}
There exists a constant $C>0$ such that for every $t>0$
\begin{eqnarray}
- \displaystyle \int_\Omega |\nabla^2 Z(\rho)|^2 -\int_\Omega\mu'(\rho)H''(\rho) |\nabla \rho|^2 
- \int_\Omega \,  \rho \,  |u|^2  \label{ineq_part1}\\
\leq- C\int_\Omega  \dfrac{\rho}{2} \left(|w|^2 + (\kappa (1-\kappa)+\varepsilon) |v|^2\right) \nonumber
\end{eqnarray}
\end{lem}

\begin{proof}
From Poincaré's inequality, we find
\begin{eqnarray*}
&& \int_\Omega |\nabla^2 Z(\rho)|^2\geq C \int_\Omega |\nabla Z(\rho)|^2 \geq C\int_\Omega \frac{\mu(\rho)}{\rho}\rho |v|^2. 
\end{eqnarray*}
Moreover,
$$
\mop H''(\rho)|\nabla \rho|^2=\dfrac{1}{4} \frac{\rho}{\mu'(\rho)}H''(\rho)\rho|v|^2,
$$
hence there exists $C>0$ such that
\begin{eqnarray*}
&&- \displaystyle \int_\Omega |\nabla^2 Z(\rho)|^2 -  \int_\Omega\mu'(\rho)H''(\rho) |\nabla \rho|^2  \leq -C \int_\Omega \left( \frac{\mu(\rho)}{\rho}+\frac{\rho H''(\rho)}{\mu'(\rho)}\right )\rho |v|^2.
\end{eqnarray*}
If now we  prove that there exists a constant $C>0$ such that for any $\rho\geq 0$
\begin{equation}
\label{eq:ajouta}
\frac{\mu(\rho)}{\rho}+\frac{\rho H''(\rho)}{\mu'(\rho)}\geq C,
\end{equation} 
then
\begin{equation}
\label{ineq_proof_part1_1}
- \displaystyle \int_\Omega |\nabla^2 Z(\rho)|^2 -  \int_\Omega\mu'(\rho)H''(\rho) |\nabla \rho|^2\leq -C \int_\Omega \rho |v|^2.
\end{equation}


\noindent{\it Proof of \eqref{eq:ajouta}.}
To do so let us consider large and small density values, the inequality being true for intermediate values. Note that for large density values or small density values with $\alpha<1$, we only use the term $\mu(\rho)/\rho$ to obtain the inequality \eqref{eq:ajouta}. But for small density values with $\alpha \geq 1$, we have to use the term $\rho H''(\rho)/\mu'(\rho)$ to obtain \eqref{eq:ajouta}.

\smallskip

\noindent{\bf First part (large density values).} 
From \eqref{condbeta2mu} or \eqref{condbeta1mu},  for  $M>0$ big enough and  $\rho>M$, we have $\rho \, \mop\geq \mo$. Integrating the ODE gives   $\mo \geq C\rho$ for $\rho>M$.  

\vskip0.2cm
\noindent{\bf Second part (small density values).} 
We have to consider two cases to control the small values of $\rho$. If  $\alpha < 1$ in \eqref{condalphamu}, then, for  $\rho_0>0$ small enough and  $\rho<\rho_0$, we have $\rho \, \mop\geq \mo$. Integrating the ODE gives   $\mo\geq C\rho$ for all values of $\rho\geq0$. 

Now, we consider the case  $\alpha \geq 1$ in \eqref{condalphamu}. For any $\eta>0$ there exists $\rho_\eta>0$ such that for any $\rho<\rho_\eta$ we have 
$$
\rho\mop\geq(\alpha-\eta)\mo.
$$
Integrating backward in $\rho$ the ODE gives that for $\rho<\rho_\eta$:
$$
\mo\leq C\rho^{\alpha-\eta}.
$$
Using \eqref{condalphamu} again gives that 
for $\rho<\rho_\eta$:
$$
\frac{\rho H''(\rho)}{\mop}\geq C\rho^{\gamma-\alpha+\eta}.
$$
If $\gamma<\alpha$, taking $\eta>0$ small enough we have:
$$
\frac{\rho H''(\rho)}{\mop}\geq C,\qquad \mathrm{for} \  \rho<\rho_\eta.
$$ 
This gives estimate \eqref{eq:ajouta} and then \eqref{ineq_proof_part1_1}.
\cqfd

Note now that  we have 
$$
|w|^2=|u+\kappa v|^2\leq 2(|u|^2+\kappa^2 |v|^2).
$$
So 
$$
-|u|^2\leq \kappa^2 |v|^2-\frac{|w|^2}{2}.
$$
Hence, using \eqref{ineq_proof_part1_1},
\begin{eqnarray}\nonumber
&&- \displaystyle \int_\Omega |\nabla^2 Z(\rho)|^2 -\int_\Omega\mu'(\rho)H''(\rho) |\nabla \rho|^2 - \int_\Omega \rho |u|^2\\ 
\nonumber
&&\qquad \leq -C \,  (\int_\Omega \rho |v|^2+ \int_\Omega \rho |u|^2)\\
\nonumber
&& \qquad \leq - C\int_\Omega  \rho \,  \left(|w|^2 + (\kappa (1-\kappa)+\varepsilon) |v|^2\right),  \label{rhov}
\end{eqnarray}
which concludes the proof of Lemma \ref{Result_part1}.
\end{proof}

\begin{rem}
Note that we can obtain the result of Lemma \ref{Result_part1} without the capillarity term and then the exponential time decay but we then need conditions on the exponent $\gamma$ of the pressure. In other terms, the use of the capillarity term allow to obtain the exponential decay for all exponent $\gamma$ of the pressure law.   
\end{rem}

\subsection{Second step of proof of proposition \ref{prop:estregsol}}
The goal now is to prove the following lemma. 
\begin{lem} \label{propnew}
There exists a constant $C>0$ such that
\begin{equation}
\label{eq_modulatedHs}
\ -  \into |\nabla^2 Z(\rho) |^2 - \into H''(\rho)\mu'(\rho)|\nabla\rho|^2 \leq -C \, \left(  \into    h(\rho|r)+\into H(\rho|r) \right).
\end{equation}
\end{lem}
To prove this lemma, we use two results: lemmas \ref{lem horrible} and \ref{lemK}. Note that lemma \ref{lem horrible} will be used to prove lemma \ref{lemK} which gives lemma \ref{propnew}.

\begin{lem}\label{lem horrible}
For any constant $\eta>0$ small enough,  there exist $0<\rho_0<r<M$ and  $C>0$ such that 
\begin{eqnarray}\label{first truc}
 \into|\nabla^2 Z|^2&\geq&C\into|\rho-r|^2 \mathbf{1}_{\{  \rho_0\leq \rho\leq M \}} \\
 && +C\into \mathbf{1}_{\{  \rho<\rho_0 \}}+C\into \rho^{3(\beta-\eta)-2} \mathbf{1}_{\{  \rho\geq M \}}.
  \nonumber
\end{eqnarray}
Moreover, if $\beta\ge 1$  then
\begin{equation}\label{second truc}
\into H''(\rho)\mop|\nabla\rho|^2\geq C \into H(\rho)  \mathbf{1}_{\{  \rho\geq M \}}.
\end{equation}
\end{lem}

\begin{proof} 
The proof is divided into three parts: The first deals with the large values of $\rho$, the second one with the small values of $\rho$, and finally the third with 
the intermediate values of $\rho$.

\vskip0.2cm

\noindent{\bf First part (large density values).} 
In this part we consider the large values of $\rho$. We first focus on the proof of \eqref{first truc} and then, using a generalization of the classical  log-Sobolev inequality we prove \eqref{second truc}.

\vskip0.2cm
\noindent{\underline{ Proof of \eqref{first truc}}}:
Since the function $\rho\mapsto Z(\rho)$ is an increasing function, we have $|\{Z(\rho)\geq Z(2r)\}|=|\{\rho\geq 2r\}|$. By the Tchebychev's inequality, we get 
\begin{equation}\label{measrhogrand}
|\{Z(\rho)\geq Z(2r)\}|=|\{\rho\geq 2r\}|\leq \frac{1}{2r}\into \rho= \frac{1}{2}.
\end{equation}
Since $|\Omega|=1$, using \eqref{measrhogrand} we have 
$$
|\{Z(\rho)< Z(2r)\}|\geq \frac{1}{2}.
$$
Using Lemma \ref{lemtech} given in Appendix on $(Z(\rho)-Z(2r))_+$ we get 
\begin{equation}
\label{eq:ajoutb}
\into(Z(\rho)-Z(2r))_+^2\leq  C\into |\nabla Z|^2.
\end{equation}
Thanks to assumptions \eqref{condbeta2mu} or \eqref{condbeta1mu}, the function $\rho \mapsto Z(\rho)$ converges to infinity when $\rho$ goes to infinity. So there exists $M>0$ such that for any $\rho\geq M$
\begin{equation}
\label{eq:ajoutc}
Z(\rho)^2\leq 2(Z(\rho)-Z(2r))_+^2.
\end{equation}
Moreover for any $\eta>0$, there exists a possibly bigger $M>0$ such that we have for any $\rho>M$:
$$
Z(\rho)^2\geq C\rho^{3(\beta-\eta)-2}.
$$
Therefore with \eqref{eq:ajoutb} and \eqref{eq:ajoutc}, we have:
$$
\into \rho^{3(\beta-\eta)-2}\mathbf{1}_{\{\rho\geq M\}}\leq C\into |Z(\rho)|^2\mathbf{1}_{\{\rho\geq M\}}\leq C\into |\nabla Z|^2.
$$

\vskip0.2cm
\noindent{\underline{Proof of \eqref{second truc}}}:
If $\gamma>1$, working with the function $J(\rho)$ instead of $Z(\rho)$ where $J(\rho)$ is defined such that 
$$
J'(\rho)=\sqrt{H''(\rho)\mu'(\rho)},
$$
with similar arguments, we  obtain \eqref{second truc}. 

If $\gamma=1$, then $H(\rho)=\rho\ln\rho$, and we need a generalization of the classical log-Sobolev inequality.
Let us denote $f=\sqrt{\rho}$. As in \eqref{measrhogrand}, we have 
\begin{equation}
|\{f\leq \al \}|\geq |\{\rho\leq 2r\}|\geq \frac{1}{2}.
\end{equation}
So, from Sobolev's injections and Lemma \ref{lemtech}, we obtain
\begin{eqnarray*}
&&\|1+(f-\al)^2_+\|_{L^3}\leq 1+\|(f-\al)^2_+\|_{L^3}\leq 1+\|(f-\al)_+\|^2_{L^6}\\
&&\qquad \leq 1+C\left(\|\nabla(f-\al)_+\|^2_{L^2}+\|(f-\al)_+\|^2_{L^2} \right)\\
&&\qquad \leq 1+C\|\nabla(f-\al)_+\|^2_{L^2}.
\end{eqnarray*}
Denoting by
$$
m=\into  1+(f-\al)^2_+\, \qquad ds=\frac{1+(f-\al)^2_+}{m},
$$
(note that $s$ is a positive measure  of mass 1), the logarithmic function being increasing and concave, using Jensen's inequality, we have
\begin{eqnarray}
\nonumber
&&C\|\nabla (f-\al)_+\|^2_{L^2}\geq \ln (1+C\|\nabla(f-\al)_+\|^2_{L^2})\\
\nonumber
&&\qquad\geq\frac{1}{3}\ln\left[ \into \left( 1+(f-\al)_+^2\right)^3\,dx\right]
=\frac{1}{3}\ln\left[  \into m\left( 1+(f-\al)_+^2\right)^2 \,ds\right] \\
&&\quad\qquad \geq \into \frac{1}{3}\ln \left( m  \left(1+(f-\al)_+^2\right)^2\right)\, ds \label{Jensen}\\
&&\quad\qquad=  \frac{2}{3}\into \ln\left(\sqrt{m}  (1+(f-\al)_+^2)\right)\frac{1+(f-\al)^2_+}{m} \, dx \nonumber \\
&&\quad\qquad \geq \frac{2}{3(1+r)}\into \ln\left( (1+(f-\al)_+^2)\right) (1+(f-\al)^2_+) \,dx, \nonumber
\end{eqnarray}
where we have used that  $1\leq m\leq 1+r$. 

Defining $M=\sup(8r, 16)$, for $f=\sqrt{\rho}\geq \sqrt{M}$, we have both
\begin{eqnarray*}
&&1+(f-\al)_+^2\geq \frac{\rho}{4} \qquad \hbox{and} \qquad \ln (1+(f-\al)_+^2)\geq \frac{1}{2}\ln \rho.
\end{eqnarray*}
Therefore, \eqref{Jensen} gives the following control
$$
C\into H(\rho) \mathbf{1}_{\{\rho>M\}}=C\into \rho\ln \rho \mathbf{1}_{\{\rho>M\}} \leq C\|\nabla (f-\al)_+\|^2_{L^2}.
$$
From Properties \eqref{condbeta2mu} and \eqref{condbeta1mu} 
$$
\sup_{\rho\geq 2r}\frac{1/\rho}{H''(\rho)\mop}\leq C,
$$
and so 
$$
C\into H(\rho) \mathbf{1}_{\{\rho>M\}}\leq C\|\nabla (f-\al)_+\|^2_{L^2}\leq C \|\nabla\sqrt{\rho} \mathbf{1}_{\{\rho\geq 2r\}}\|^2_{L^2}\leq C\into H''(\rho)\mop|\nabla\rho|^2
$$
which gives \eqref{second truc}.

\vskip0.2cm 
\noindent{\bf Second part (small density values).} In this part we consider the small values of $\rho$.
From \eqref{condalphamu}, there exists $\sigma>0$ such that 
$$
Z'(\rho)\geq C\rho^{\sigma-1},
$$
and so
\begin{equation}
\label{ineq_gradrhosigma}
\into |\nabla \rho^\sigma|^2\leq C \into|\nabla Z|^2.
\end{equation}
Let us fix $\rho_0=\dfrac{r}{4}$ and $\delta=\inf \left( \dfrac{r}{16}, \dfrac{1}{2} \right)$. Note that, if $C\displaystyle \into|\nabla^2 Z|^2\geq \delta$, 
$$
\into {\bf 1}_{\{\rho\leq \rho_0\}}\leq 1\leq \frac{C}{\delta} \into|\nabla^2 Z|^2, 
$$
which immediately gives \eqref{first truc} for $\rho<\rho_0$.
Then, without loss of generality, we can assume that 
$$
C\into|\nabla^2 Z|^2\leq \delta,
$$
such that the result {\color{blue}of the first part of the proof} implies
$$
\into {\bf 1}_{\{\rho\geq M\}}(\rho-r)_+\leq  M^{-3(\beta-\eta)+3}\into\rho^{3(\beta-\eta)-2}  {\bf 1}_{\{\rho\geq M\}} \leq C\into|\nabla^2 Z|^2\leq\delta.
$$
We claim that 
\begin{equation}\label{claim}
|\{\rho>r/2\}|\geq \frac{\delta}{M+1}.
\end{equation}
To prove this claim let us show that if it is not the case we obtain a contradiction. 
Assuming
$$|\{\rho>r/2\}| < \frac{\delta}{M+1},$$
 we have
$$
M|\{\rho\geq r\}|\leq M|\{\rho > r/2\}| \leq \delta,
$$
and so
$$
\into(\rho-r)_+\leq \into {\bf 1}_{\{\rho\geq M\}}(\rho-r)_++M|\{\rho\geq r\}|\leq 2\delta.
$$
Since $\displaystyle r=\into\rho$, we have $\displaystyle \into(\rho-r)_+=\into(r-\rho)_+$.
Using Tchebychev's inequality, we get 
\begin{eqnarray*}
&&\qquad |\{\rho\leq r/2\}|=|\{(r-\rho)_+\geq r/2\}|\leq \frac{2}{r}\into(r-\rho)_+\leq \frac{4\delta}{r}\leq \dfrac{1}{4}.
\end{eqnarray*}
Therefore 
$$
|\{\rho>r/2\}|\geq 1-1/4=\dfrac{3}{4},
$$
which is a contradiction, since $\dfrac{\delta}{M+1} \leq \delta\leq \dfrac{1}{2}$. Therefore the claim \eqref{claim} is proved.

\vskip0.2cm
Let us now rewrite the claim  \eqref{claim} (since $\sigma>0$)
$$
|\{(r/2)^\sigma-\rho^\sigma<0\}|\geq \frac{\delta}{M+1}.
$$
Since, $|\{((r/2)^\sigma-\rho^\sigma)_+=0\}|=|\{(r/2)^\sigma-\rho^\sigma<0\}|\neq 0$, using Lemma \ref{lemtech} with the function $((r/2)^\sigma-\rho^\sigma)_+$,  \eqref{ineq_gradrhosigma} and Poincaré's inequality, we get
\begin{eqnarray*}
\into((r/2)^\sigma-\rho^\sigma)^2_+\leq C\into|\nabla((r/2)^\sigma-\rho^\sigma)_+|^2 \leq C\into|\nabla\rho^\sigma|^2\leq C \into|\nabla Z|^2\leq C \into|\nabla^2 Z|^2.
\end{eqnarray*}
Using the Tchebychev's inequality and the previous one we get
$$
|\{\rho\leq r/4\}| \leq \dfrac{C}{(2^\sigma-1)^2}\left(\frac{4}{r}\right)^{2\sigma} \into|\nabla^2 Z|^2.
$$
This gives \eqref{first truc} for $\rho<\rho_0$.

\vskip0.2cm
\noindent{\bf Third part (Intermediate density values).} In this part we deal with the intermediate values of $\rho$ {\it i.e.} $\rho_0 \leq \rho \leq M$. Let us introduce
$$
\rt(t,x)=\sup(\rho_0,\inf(M,\rho(t,x))).
$$
Note that $|Z'(\rho)|$ is bounded by below on $[\rho_0,M]$, hence
$$
|\nabla\rt|=|\nabla\rho| \mathrm{1}_{\{ \rho_0\leq \rho\leq M\}}=\left|\frac{\nabla Z(\rho)}{Z'(\rho)}\right| \mathrm{1}_{\{ \rho_0\leq \rho\leq M\}}\leq C|\nabla Z|.
$$
Therefore, using Poincar\'e's inequality, since $|\Omega|=1$ and $\rho_0=r/4$,
\begin{eqnarray*}
&&\into|\rho-r|^2{\mathbf{1}}_{\{\rho_0\leq \rho\leq M\}}\leq \into|\rt-r|^2 \\
&&\qquad \qquad  \leq 2 \left[ \into\left|\rt-\into\rt\right|^2+\into\left|\into\rt-r\right|^2 \right]\\
&&\qquad \qquad \leq C \left[ \into|\nabla Z|^2+\left(\into\rt-\into\rho\right)^2 \right]\\
&&\qquad \qquad =C \left[ \into|\nabla Z|^2+\left(\int_{\{\rho>M\}}(M-\rho)+\int_{\{\rho<\rho_0\}}(\rho_0-\rho)\right)^2 \right] \\
&&\qquad \qquad \leq C \left[ \into|\nabla Z|^2+r\left(\int_{\{\rho>M\}}\rho\right)+\dfrac{r}{4}\left(\int_{\{\rho<\rho_0\}}\rho_0\right) \right]\\
&&\qquad \qquad \leq C \left[ \into|\nabla Z|^2+r\left(M^{-3(\beta-\eta)+3}\into\mathbf{1}_{\{\rho>M\}}\rho^{3(\beta-\eta)-2}\right)+\dfrac{\rho_0r}{4}\into \mathbf{1}_{\{\rho<\rho_0\}} \right].
\end{eqnarray*}
We get the desired result using the first and second parts of the proof.
\end{proof}
As mention before to complete the proof of lemma \ref{propnew}, we need also this following lemma:
\begin{lem}\label{lemK}
Assume that $G$ is a function of $\rho$ verifying $G$ is $C^2$ on $[\rho_0,M]$, and that there exists a constant $C>0$ such that 
\begin{eqnarray}\label{hypK1}
&& G(\rho)\leq C\sup(\rho^{3(\beta-\eta)-2},H(\rho)),\qquad \mathrm{for \ \ }\rho\geq M,\\
\label{hypK2}
&&G(\rho)\leq C,\qquad  \mathrm{for \ \ }\rho\leq \rho_0.
\end{eqnarray}
Then, there exists an other constant $C>0$ such that 
$$
\into G(\rho|r)\leq C\into|\nabla^2 Z|^2+C\into H''(\rho)\mu'(\rho)|\nabla\rho|^2,
$$
where $G(\rho|r)$ is given by
$$
G(\rho|r)=G(\rho)-G(r)-G'(r)(\rho-r).
$$
\end{lem}
\begin{proof}
We have 
\begin{eqnarray*}
&&\into G(\rho|r)=\into G(\rho|r)\mathbf{1}_{\{\rho_0\leq\rho\leq M\}}+\into G(\rho|r)\mathbf{1}_{\{\rho<\rho_0\}}+\into G(\rho|r)\mathbf{1}_{\{\rho> M\}}\\
&&\qquad \leq C\left(\sup_{[\rho_0,M]} |G''|\right)\into |\rho-r|^2\mathbf{1}_{\{\rho_0\leq\rho\leq M\}} \\
&&\hskip3cm  +C\into\mathbf{1}_{\{\rho<\rho_0\}}+C\into \sup(\rho^{3(\beta-\eta)-2},H(\rho))\mathbf{1}_{\{\rho> M\}}.
\end{eqnarray*}
Lemma \ref{lem horrible} finishes the proof.
\end{proof}

To finish the proof of lemma \ref{propnew}, the idea is now to apply Lemma \ref{lemK} with $G$ equals $h$ and $H$. Note that \eqref{hypK2} is valid for both quantities in all configurations. We want to show that \eqref{hypK1} is also true for both quantities in all configurations. To this end, we consider two  cases. 

\vskip0.2cm
\noindent
{\bf Case 1 $\gamma >1$.}
Since $\beta >1$, for $\eta$ small enough,
\begin{equation}\label{betabusiness}
1<\beta +\eta<3(\beta-\eta)-2.
\end{equation}
Solving the ODE \eqref{condbeta2mu} for $\rho>r$, we find that there exists a constant $C>0$ such that for any $\rho > r$:
$$
\mu(\rho )\leq C\rho^{\beta+\eta},
$$
and using again \eqref{condbeta2mu} to get a bound by above on $\mu'$ from the bound on $\mu$, we get 
$$
\mu'(\rho)\leq C \rho^{\beta+\eta-1}.
$$
Since $h''(\rho)=\dfrac{\mu'(\rho)}{\rho}$, using the bound above on $\mu'$ and integrating in $\rho$ twice, we find
 that  there exists a constant $C>0$ such that for  every $\rho>M$, we have  
$$ 
h(\rho)\leq C\rho^{\beta+\eta},
$$
which, thanks to \eqref{betabusiness}, leads to
$$
h(\rho)\leq C\sup(\rho^{3(\beta-\eta)-2},\rho^\gamma).
$$
This proves \eqref{hypK1} for $h$. 

Since $\gamma>1$, then obviously, for $\rho$ big enough, we have
$$
H(\rho)=\frac{\rho^\gamma}{\gamma-1} \leq C \rho^\gamma.
$$
\vskip0.3cm
\noindent
{\bf Case 2 $\gamma=1$.}
Then we have $h(\rho)=H(\rho)=\rho\ln\rho$. Since $3(\beta-\eta)-2> 1$, for $\rho$ big enough
$$
h(\rho)=H(\rho)=\rho \ln\rho\leq \rho^{3(\beta-\eta)-2}
$$
and $H$ satisfies \eqref{hypK1}. 
\vskip0.3cm
Then in all cases, we can apply Lemma \ref{lemK} to find 
$$
\into 2 \,  \kappa \,  r_2 \,  h(\rho |r)+\into H(\rho |r)\leq C \left(\into|\nabla^2 Z(\rho)|^2+\into H''(\rho)\mu'(\rho)|\nabla\rho|^2\right),
$$
which ends the proof of lemma \ref{propnew}.

\subsection{Ends of proof of proposition \ref{prop:estregsol} and theorem \ref{th:exponentialdecay}}
We can now state and prove the proposition which is the goal of this section since it allows us to obtain the exponential time decay given in Theorem \ref{th:exponentialdecay} letting
$\zeta$ tends to zero if we obtain the exponential decay for $\mathcal{E}_\zeta$
with constant independent on $\zeta$. Here we adopt again the indice $\zeta$ for
reader's convenience.

Let us recall inequality \eqref{dissip} of Definition \ref{defweak1}:
\begin{equation*}
\dfrac{d\mathcal{E}_\zeta}{dt}(t) \leq  - \varepsilon  \, C_{\widetilde\nu,\alpha_1} \, \displaystyle \int_\Omega |\nabla^2 Z(\rho_\zeta)|^2 - \kappa \int_\Omega\mu'(\rho_\zeta)H''(\rho_\zeta) |\nabla \rho_\zeta|^2 
 - \int_\Omega  r_2 \,   \rho_\zeta \,  |u_\zeta|^2.
\end{equation*}
Applying Lemmas \ref{Result_part1} and \ref{propnew}, we obtain:
\begin{equation*}
\dfrac{d\mathcal{E}_\zeta}{dt}(t) \leq -C \int_\Omega \biggl[\rho_\zeta \left(|w_\zeta|^2+(\kappa(1-\kappa)+\varepsilon)|v_\zeta|^2\right) 
+H(\rho_\zeta |r) + 2\kappa r_2 h(\rho_\zeta | r)\biggr], 
\end{equation*}
with $C$ a constant which only depends on $\varepsilon, \tilde \nu, \alpha_1, \alpha_2, \kappa$ and $r_2$, but which is independent of $\zeta$. 
Now using the definition of $\mathcal{E}_\zeta$ given in \eqref{Energy} we finally obtain:
\begin{eqnarray}
\dfrac{d\mathcal{E}_\zeta}{dt}(t) &\leq&  -C\mathcal{E}_\zeta(t) . \label{dissipfin2}
\end{eqnarray}
Then using the Gronwall Lemma, we obtain
\begin{eqnarray}
\mathcal{E_\zeta}(t) &\leq&  \mathcal{E_\zeta}(0) \exp\left({-Ct}\right), \label{dissipfin3}
\end{eqnarray}
which directly gives the result of Theorem \ref{ineqentrel}. Then, since $\mathcal{E}_\zeta$ is lower semicontinuous and since we assume the strong convergence initially, passing to the limit $\zeta$ tends to zero leads to 
\begin{eqnarray*}
\mathcal{E}(t) &\leq&  \mathcal{E}(0) \exp\left({-Ct}\right), 
\end{eqnarray*}
which proves Theorem \ref{th:exponentialdecay}.


\section{Appendix}
\vskip0.5cm
Let us prove the following classical technical Lemma.

\begin{lem}\label{lemtech}
For any function $f\geq0$ such that $|\{f=0\}|\geq \delta>0$, there exists a constant $C$ depending on $\delta$ and $\Omega $, such that 
$$
\into f^2\,dx\leq C \into|\nabla f|^2\,dx.
$$ 
\end{lem}

\begin{proof}
From the fundamental theorem of calculus, for any $x,y\in \Omega$
$$
f(x)-f(y)=\int_0^1 \nabla f(tx+(1-t)y)\cdot(x-y)\,dt.
$$
Therefore
\begin{eqnarray*}
&&\into |f(x)|^2\,dx\leq\frac{1}{\delta} \int_{\{f=0\}}\left(\into |f(x)|^2\,dx\right)\,dy= \frac{1}{\delta}\int_{\{f=0\}}\into |f(x)-f(y)|^2\,dx\,dy\\
&&\qquad\leq  \frac{1}{\delta}\into\into |f(x)-f(y)|^2\,dx\,dy\leq  \frac{1}{\delta}\into\into\int_0^1 |\nabla f(tx+(1-t)y)|^2|x-y|^2\,dx\,dy\,dt\\
&&\qquad \leq C\into\into\left(\int_0^{1/2} |\nabla f(tx+(1-t)y)|^2\,dt+\int_{1/2}^{1} |\nabla f(tx+(1-t)y)|^2\,dt  \right)\,dx\,dy\\
&&\qquad \leq 2C\into\into\int_0^{1/2} |\nabla f(tx+(1-t)y)|^2\,dt\,dx\,dy\\
&&\qquad \leq 2C\into\int_0^{1/2} \left(\into|\nabla f(tx+(1-t)y)|^2\,dy\right)\,dt\,dx\\
&&\qquad \leq 2C\into\int_0^{1/2} \left(\int_{(1-t)\Omega}|\nabla f(tx+z)|^2\,\frac{dz}{(1-t)^3}\right)\,dt\,dx\\
&&\qquad \leq 2C\into\int_0^{1/2} \left(\int_{\Omega}|\nabla f(tx+z)|^2\,\frac{dz}{(1-t)^3}\right)\,dt\,dx\\
&&\qquad \leq 16C\into\int_0^{1/2} \left(\into|\nabla f(z)|^2\,dz\right)\,dt\,dx\\
&&\qquad \leq 8C\into|\nabla f(z)|^2\,dz.
\end{eqnarray*}

\end{proof}

\section*{Acknowledgements}
The third author acknowledges support from the team INRIA/RAPSODI and the Labex CEMPI (ANR-11-LABX-0007-01). The first author acknowledges the project  TELLUS INSU-INSMI "Approche crois\'ee pour fluides visco-\'elasto-plastiques: vers une meilleure compr\'ehension des zones solides/fluides". The first and the second authors acknowledge the ANR Project FRAISE managed by C. {\sc Ruyer-Quil}. The fourth author was partially supported by the NSF, DMS 1614918.


\end{document}